\documentclass[10pt]{article}

\usepackage{amsmath,amsthm}
\usepackage{amssymb}
\usepackage{hyperref}
 
\allowdisplaybreaks

\begin{document}

\title{Cameron's operator in terms of determinants, and hypergeometric numbers 
}  
\author{Narakorn Rompurk Kanasri
\\   
\small Department of Mathematics\\[-0.8ex]
\small Khon Kaen University, Khon Kaen, 40002, Thailand\\[-0.8ex]
\small \texttt{naraka@kku.ac.th}\\\\
Takao Komatsu
\\   
\small Department of Mathematical Sciences, School of Science\\
\small Zhejiang Sci-Tech University\\
\small Hangzhou 310018, China\\[-0.8ex]
\small \texttt{komatsu@zstu.edu.cn}\\\\
Vichian Laohakosol
\\   
\small Department of Mathematics, Faculty of Science\\[-0.8ex]
\small Kasetsart University, Bangkok, 10900, Thailand\\[-0.8ex]
\small \texttt{fscivil@ku.ac.th}
}

\date{
}

\maketitle

\def\fl#1{\left\lfloor#1\right\rfloor}
\def\cl#1{\left\lceil#1\right\rceil}
\def\ang#1{\left\langle#1\right\rangle}
\def\stf#1#2{\left[#1\atop#2\right]} 
\def\sts#1#2{\left\{#1\atop#2\right\}}

\newtheorem{theorem}{Theorem}
\newtheorem{Prop}{Proposition}
\newtheorem{Cor}{Corollary}
\newtheorem{Lem}{Lemma}

\begin{abstract}
By studying Cameron's operator in terms of determinants, two kinds of ``integer'' sequences of incomplete numbers were introduced.  One was the sequence of restricted numbers, including $s$-step Fibonacci sequences.  Another was the sequence of associated numbers, including Lam\'e sequences of higher order.  
By the classical Trudi's formula and the inverse relation, more expressions were able to be obtained.       
These relations and identities can be extended to those of sequence of negative integers or rational numbers. As applications, we consider hypergeometric Bernoulli, Cauchy and Euler numbers with some modifications. 
\end{abstract}

\allowdisplaybreaks

\section{Introduction}

In 1989, Cameron \cite{Cameron} considered the operator $A$ defined on the set of sequences of non-negative integers as follows: for $x=\{x_n\}_{n\ge 1}$ and $z=\{z_n\}_{n\ge 1}$, set $A x=z$, where 
\begin{equation}  
1+\sum_{n=1}^\infty z_n t^n=\left(1-\sum_{n=1}^\infty x_n t^n\right)^{-1}\,.
\label{cameron}
\end{equation}   
Suppose that $x$ enumerates a class $C$. Then $A x$ enumerates the class of disjoint unions of members of $C$, where the order of the "component" members of $C$ is significant. The operator $A$ also plays an important role for free associative (non-commutative) algebras. More motivations and background together with many concrete examples (in particular, in the aspects of Graph Theory) by this operator can be seen in \cite{Cameron}.  
There is a similarly general method based on Bourbaki's formula for the derivation of $n$-th order of $f(x)$ divided by $g(x)$ introduced in \cite{QG,QKD}. 

There are many examples in Combinatorics.  For example, if the sequence of the numbers $x_n$ is in arithmetic progression, the numbers $z_n$ are yielded from the three-term recurrence relation, in terms of the Cameron's operator by (\ref{cameron}).  That is, when $x_n=(n-1)a+b$ ($n\ge 1$) then the numbers $z_n$ satisfy the recurrence relation 
\begin{equation} 
z_n=(b+2)z_{n-1}+(a-b-1)z_{n-2}\quad(n\ge 3),\quad z_2=a+b(b+1),\quad z_1=b\,, 
\label{example1} 
\end{equation}
because the identity in (\ref{cameron}) is equal to 
$$
\frac{1-2 t+t^2}{1-(b+2)t-(a-b-1)t^2}\,.
$$ 
In particular, if $a=1$ and $b=2$, $z_n$ is yielded from the sequence 
$$
\{x_n\}_{n\ge 1}=2,3,4,5,6,7,\dots\,,  
$$   
and the number of $L$-convex polyominoes with $n$ cells, that is convex polyominoes where any two cells can be connected by a path internal to the polyomino and which has at most $1$ change of direction (see also \cite{CFRR,CFMRR} and \cite[A003480]{oeis}).   

In \cite{Charalambides}, two kinds of generalized Stirling numbers of both kinds are introduced in Combinatorial interpretations.   The Stirling numbers of the second kind enumerate the number
of partitions of a set with $n$ elements into $k$ non-empty blocks. The Stirling numbers can be generalized by using a restriction on the size of the blocks and cycles. In particular,
the {\it restricted} Stirling numbers of the second kind give the number of partitions of $n$ elements into $k$ subsets, with the additional restriction that none of the blocks contain more than $m$ elements. The {\it associated} Stirling numbers of the second kind equals the number that each block contains at most $m$ elements.  These numbers are called {\it incomplete} Stirling numbers of the second kind, together.  The (unsigned) Stirling numbers of the first kind enumerate the number of permutations on $n$ elements with $k$ cycles.  The {\it restricted} Stirling numbers of the first kind and the {\it associated} Stirling numbers of the first kind are similarly defined.  By using these restricted and associated numbers, several generalizations of Bernoulli and Cauchy numbers are introduced and studied in \cite{Ko8,KLM,KMS,KR0}.  

In this paper, by applying the similar concept of incomplete Stirling numbers, we study more general operators.  One type includes Fibonacci $s$-step numbers, satisfying the recurrence relation $F_n^{(s)}=F_{n-1}^{(s)}+F_{n-2}^{(s)}+\cdots+F_{n-s}^{(s)}$, as special cases. Another includes the numbers in Lam\'e sequence of higher order or Fibonacci $p$-numbers (\cite{KS}), satisfying the recurrence relation $\mathfrak L_n=\mathfrak L_{n-1}+\mathfrak L_{n-s}$, as special cases.  
By the classical Trudi's formula and the inverse relation, more expressions can be obtained.       
These relations and identities are not restricted to the nonnegative integer sequences, but can be extended to those of sequence of negative integers or rational numbers. As applications, we consider hypergeometric Bernoulli, Cauchy and Euler numbers with some modifications.

\section{Incomplete numbers with their expressions}  

Given two integer\footnote{In the applications' section, rational sequences will be considered.} (finite or infinite) sequences $x=\{x_1,x_2,\dots\}$ and $z=\{z_1,z_2,\dots\}$.  We assume that $x_0=1$ and $z_0=1$ unless we specify.  
We shall consider a more general operator $\mathcal A x=z$, where for a positive integer $m$,   
\begin{equation}  
1+\sum_{n=1}^\infty z_n t^n=\left(1-\sum_{n=1}^m x_n t^n\right)^{-1}\,.
\label{cameron-m}
\end{equation} 
When $m\to\infty$, the relation (\ref{cameron-m}) is reduced to the relation (\ref{cameron}).  On the contrary to the {\it associated numbers} in Section \ref{a-numbers}, the numbers $z_n$ may be called {\it restricted numbers}, yielding from the sequence of the numbers $x_n$ in terms of the operator in (\ref{cameron-m}). 
For example, in the case where $\{x_n\}_{n=1}^m=\{\underbrace{1,\dots,1}_m\}$, for $m=2$, $\mathcal A\{x_n\}$ yields the Fibonacci sequence $\{z_n\}=\{F_{n+1}\}$, where $F_n=F_{n-1}+F_{n-2}$ ($n\ge 2$) with $F_1=F_2=1$.  For $m=3$, $\mathcal A\{x_n\}$ yields the Tribonacci sequence $\{z_n\}=\{T_{n+1}\}$, where $T_n=T_{n-1}+T_{n-2}+T_{n-3}$ ($n\ge 3$) with $T_1=T_2=1$ and $T_3=2$.  Therefore, the combinatorial interpretation is that $z_n=T_{n+1}$ is the number of ways of writing $n$ as an ordered sum of $1$'s, $2$'s and $3$'s.  For example, $T_5=7$ because $n=4$ can be expressed in $7$ ways as 
$$
1+3=3+1=2+2=1+1+2=1+2+1=2+1+1=1+1+1+1\,. 
$$ 

By using a recurrence relation   
\begin{equation} 
z_n=\sum_{k=1}^{\min\{n,m\}}x_k z_{n-k}\quad(n\ge 1)\,,    
\label{rec:har}  
\end{equation}   
the number $z_n$ can be expressed in terms of the determinant \cite[Lemma 1]{Ko14}.    

\begin{Lem}  
For integers $1\le n\le m$, 
\begin{equation} 
z_n=\left|
\begin{array}{ccccc} 
x_1&1&0&&\\  
-x_2&x_1&&&\\ 
\vdots&\vdots&\ddots&1&0\\ 
(-1)^n x_{n-1}&(-1)^{n-1}x_{n-2}&\cdots&x_1&1\\ 
(-1)^{n-1}x_n&(-1)^{n}x_{n-1}&\cdots&-x_2&x_1 
\end{array} 
\right|\,.  
\label{det1}
\end{equation}  
For integers $n\ge m$,  
\begin{equation} 
z_n=\left|\begin{array}{cc} 
\underbrace{ 
\begin{array}{ccc}
x_1&1&0\\
\vdots&&\ddots\\
(-1)^{m-1}x_m&&\\
0&\ddots&\\ 
&&\ddots\\
&&0
\end{array}
}_{n-m} 
\underbrace{ 
\begin{array}{ccc}
&&\\
\ddots&&\\
&\ddots&\\
&\ddots&0\\
&&1\\
(-1)^{m-1}x_m&\cdots&x_1
\end{array}
}_{m} 
\end{array} 
\right|\,. 
\label{det2}  
\end{equation}
\label{det:har}    
\end{Lem}

\noindent 
{\it Remark.}  
For integers $n\ge 3$, the $n$-th Fibonacci number $F_n$ can be expressed as 
$$ 
F_n=\left|\begin{array}{cc} 
\underbrace{ 
\begin{array}{cccc}
1&1&0&\\
-1&1&1&\\
0&-1&\ddots&\ddots\\
&&\ddots&\ddots\\
&&&-1\\
0&&&0 
\end{array}
}_{n-3} 
\begin{array}{cc} 
&\\ 
&\\
&\\
&\\
1&0\\
1&1\\
-1&1
\end{array}
\end{array} 
\right| 
$$ 
(\cite[365 (p. 54)]{Proskuryakov}\footnote{In this book Fibonacci numbers do not begin from $1,1,2,3,\dots$ but from $1,2,3,5,\dots$}). 
For negative indices,  
$$ 
F_{-(n+1)}=\left|\begin{array}{cc} 
\underbrace{ 
\begin{array}{cccc}
-1&1&0&\\
-1&-1&1&\\
0&-1&\ddots&\ddots\\
&&\ddots&\ddots\\
&&&-1\\
0&&&0 
\end{array}
}_{n-3} 
\begin{array}{cc} 
&\\ 
&\\
&\\
&\\
1&0\\
-1&1\\
-1&-1
\end{array}
\end{array} 
\right| 
$$ 
were studied in \cite{KT}.  
Recently, Fibonacci numbers in the form $\pm F_{\frac{n-k}{a}+l}$ with determinants of tridiagonal matrices with $1$'s on the superdiagonal and subdiagonal and alternating blocks of $1$'s and $-1$'s of the length a, where $a$ is any positive integer and $a\ne 0\pmod 3$, studied in \cite{Troj}. 
\bigskip 
  
By using a combinatorial summation, the number $z_n$ has an explicit expression \cite[Theorem 1]{Ko14}.  

\begin{Lem}  
For integers $n\ge 1$, 
$$
z_n=\sum_{k=1}^{n}\sum_{i_1+\cdots+i_k=n\atop 1\le i_1,\dots,i_k\le m}x_{i_1}\cdots x_{i_k}\,. 
$$ 
\label{ex:har} 
\end{Lem}

By applying Trudi's formula \cite{Brioschi,Trudi} (see also \cite[Vol.3, pp. 208--209, p. 214]{Muir}), to Lemma \ref{det:har}, we obtain an explicit expression for numbers $z_n$ in terms of $x_n$ with multinomial coefficients \cite[Theorem 4]{Ko14}. 

\begin{Lem}  
For $n\ge 1$, we have 
$$ 
z_n=\sum_{t_1+2 t_2+\cdots+m t_m=n}\binom{t_1+\cdots+t_m}{t_1,\dots,t_m}
x_1^{t_1}\cdots x_m^{t_m}\,. 
$$ 
\label{th300}  
\end{Lem}

By applying the inversion relation (see, e.g. \cite{KR}): 
\begin{align}
&\sum_{k=0}^n(-1)^{n-k}\alpha_k D(n-k)=0\quad(n\ge 1)\notag\\ 
&\Longleftrightarrow \notag\\ 
&\alpha_n=\begin{vmatrix} D(1) & 1 & & \\
D(2) & \ddots &  \ddots & \\
\vdots & \ddots &  \ddots & 1\\
D(n) & \cdots &  D(2) & D(1) \\
 \end{vmatrix}
\quad\Longleftrightarrow\quad 
D(n)=\begin{vmatrix} \alpha_1 & 1 & & \\
\alpha_2 & \ddots &  \ddots & \\
\vdots & \ddots &  \ddots & 1\\
\alpha_n & \cdots &  \alpha_2 & \alpha_1 \\
 \end{vmatrix}    
\label{inversion}
\end{align}  
with $\alpha_0=D(0)=1$,  from Lemma \ref{det:har}, 
we have the following relation \cite[Theorem 4]{Ko14}.  

\begin{Lem}   
For $n\ge 1$, we have 
$$ 
\left|\begin{array}{ccccc}
z_1&1&0&&\\
z_2&z_1&&&\\
\vdots&\vdots&\ddots&1&0\\
z_{n-1}&z_{n-2}&\cdots&z_1&1\\
z_{n}&z_{n-1}&\cdots&z_2&z_1 
\end{array}\right|
=\begin{cases}
(-1)^{n-1}x_n&(1\le n\le m)\\
0&(n\ge m+1)
\end{cases}\,. 
$$ 
\label{th400}  
\end{Lem}

Therefore, by the inversion relations with Lemma \ref{ex:har} and Lemma \ref{th300}, we have the following relation \cite[Theorem 5]{Ko14}.  

\begin{Lem}  
For $n\ge 1$, we have 
\begin{align*} 
&\sum_{k=1}^{n}(-1)^{k-1}\sum_{i_1+\cdots+i_k=n\atop i_1,\dots,i_k\ge 1}z_{i_1}\cdots z_{i_k}\\
&=\sum_{t_1+2 t_2+\cdots+n t_n=n}\binom{t_1+\cdots+t_n}{t_1,\dots,t_n}
(-1)^{t_1+\cdots+t_n-1}z_1^{t_1}\cdots z_n^{t_n}\\
&=\begin{cases}
x_n&(1\le n\le m)\\
0&(n\ge m+1)
\end{cases}\,. 
\end{align*} 
\label{inversion:har}
\end{Lem}

\subsection{Examples}  

For Fibonacci numbers $F_n$, we have 
\begin{equation}
F_n=\sum_{t_1+2 t_2=n-1\atop t_1,t_2\ge 0}\frac{(t_1+t_2)!}{t_1!t_2!}(-1)^{n-t_1-1} 
\label{ex:f2}
\end{equation}
and 
\begin{equation} 
\left|\begin{array}{ccccc}
F_2&1&0&&\\
F_3&F_2&&&\\
\vdots&\vdots&\ddots&1&0\\
F_{n}&F_{n-1}&\cdots&F_2&1\\
F_{n+1}&F_{n}&\cdots&F_3&F_2 
\end{array}\right|
=\begin{cases}
(-1)^{n-1}&(n=1,2)\\
0&(n\ge 3)
\end{cases}\,. 
\label{ex:f3}
\end{equation}
For example, in (\ref{ex:f2})  
$$
F_6=\frac{3!}{1!2!}+\frac{4!}{3!1!}+\frac{5!}{5!0!}=8
$$
and 
$$
F_7=\frac{3!}{0!3!}+\frac{4!}{2!2!}+\frac{5!}{4!1!}+\frac{6!}{6!0!}=13\,.
$$

For Tribonacci numbers $T_n$, we have 
\begin{align} 
&\sum_{k=1}^{n}(-1)^{k-1}\sum_{i_1+\cdots+i_k=n\atop i_1,\dots,i_k\ge 1}T_{i_1+1}\cdots T_{i_k+1}
\label{ex:t4}\\
&=\sum_{t_1+2 t_2+\cdots+n t_n=n}\binom{t_1+\cdots+t_n}{t_1,\dots,t_n}
(-1)^{t_1+\cdots+t_n-1}T_2^{t_1}\cdots T_{n+1}^{t_n}
\label{ex:t5}\\
&=\begin{cases}
1&(n=1,2,3)\\
0&(n\ge 4)
\end{cases}\notag\,. 
\end{align} 
For example, for $n=3$ in (\ref{ex:t4}) 
\begin{align*} 
&\sum_{k=1}^{3}(-1)^{k-1}\sum_{i_1+\cdots+i_k=n\atop i_1,\dots,i_k\ge 1}T_{i_1+1}\cdots T_{i_k+1}\\
&=T_4-2 T_2 T_3+T_2 T_2 T_2=4-2\cdot 1\cdot 2-1\cdot 1\cdot 1=1
\end{align*}
and in (\ref{ex:t5})  
\begin{align*} 
&\sum_{t_1+2 t_2+3 t_3=3}\binom{t_1+t_2+t_3}{t_1,t_2,t_3}
(-1)^{t_1+t_2+t_3-1}T_2^{t_1}T_3^{t_2}T_{4}^{t_3}\\
&=\frac{3!}{3!0!0!}(-1)^{3-1}1^3 2^0 4^0+\frac{2!}{1!1!0!}(-1)^{1+1-1}1^1 2^1 4^0+\frac{1!}{0!0!1!}(-1)^{1-1}1^0 2^0 4^1\\
&=1\cdot 1+(-2)\cdot 2+1\cdot 4=1\,. 
\end{align*}  
For $n=4$ in (\ref{ex:t4})  
\begin{align*} 
&\sum_{k=1}^{4}(-1)^{k-1}\sum_{i_1+\cdots+i_k=n\atop i_1,\dots,i_k\ge 1}T_{i_1+1}\cdots T_{i_k+1}\\
&=T_5-(2 T_2 T_4+T_3^2)+3 T_2^2 T_3-T_2^4\\
&=7-2\cdot 1\cdot 4-2^2+3\cdot 1^2\cdot 2-1^4=0
\end{align*}
and in (\ref{ex:t5})  
\begin{align*} 
&\sum_{t_1+2 t_2+3 t_3+4 t_4=4}\binom{t_1+t_2+t_3+t_4}{t_1,t_2,t_3,t_4}
(-1)^{t_1+t_2+t_3+t_4-1}T_2^{t_1}T_3^{t_2}T_{4}^{t_3}T_{5}^{t_4}\\ 
&=\frac{4!}{4!0!0!0!}(-1)^{4-1}1^4 2^0 4^0 7^0+\frac{3!}{2!1!0!0!0!}(-1)^{2+1-1}1^2 2^1 4^0 7^0\\
&\quad +\frac{2!}{0!2!0!0!}(-1)^{2-1}1^0 2^2 4^0 7^0+\frac{2!}{1!0!1!0!}(-1)^{1+1-1}1^1 2^0 4^1 7^0\\
&\quad +\frac{1!}{0!0!0!1!}(-1)^{1-1}1^0 2^0 4^0 7^1\\
&=-1+3\cdot 2-4-2\cdot 4+7=0\,. 
\end{align*}

\subsection{Associate numbers}\label{a-numbers} 

In \cite{Ko8,KLM,KMS,KR0}, some kinds of associated numbers are studied in terms of associated  Stirling numbers in \cite{Charalambides}.   

In general, for a positive integer $m$ the {\it associated} operator $\mathcal A_{\ge m} x=z$ is defined by the generating function 
\begin{equation}  
1+\sum_{n=1}^\infty z_n t^n=\left(1-\sum_{n=m}^\infty x_n t^n\right)^{-1}\,.
\label{cameron-am} 
\end{equation} 
When $m=1$, the relation (\ref{cameron-am}) is reduced to the relation (\ref{cameron}).  
In this sense, the operator in (\ref{cameron-m}) is called the {\it restricted} operator because of the restricted Stirling numbers in \cite{Charalambides}.  

By a recurrence relation between two sequences 
\begin{equation}
z_n=\sum_{k=0}^{n-m}x_{m+k}z_{n-m-k}\quad(n\ge 1)  
\label{rec:har-a}
\end{equation} 
with $z_0=1$ and $z_1=\cdots=z_{m-1}=0$, 
or 
\begin{equation}
z_n=x_n+\sum_{k=0}^{n-2 m}x_{m+k}z_{n-m-k}\quad(n\ge 1)\,,    
\label{rec:har-a-a}  
\end{equation}   
we have a determinant expression for $z_n$ by means of the associated operator \cite[Theorem 2]{Ko14}.  

\begin{Lem}  
For integers $n\ge m\ge 1$,  
\begin{equation} 
z_n=\left|\begin{array}{cc} 
\underbrace{ 
\begin{array}{ccc}
0&1&0\\
\vdots&0&1\\
0&&\\
(-1)^{m-1}x_m&&\\
\vdots&\ddots&0\\
(-1)^{n-1}x_n&\cdots&(-1)^{m-1}x_m
\end{array}
}_{n-m+1} 
\underbrace{ 
\begin{array}{ccc}
&&\\
&&\\
\ddots&&\\
&1&0\\
&0&1\\
0&\cdots&0
\end{array}
}_{m-1} 
\end{array} 
\right|\,. 
\label{det-a}  
\end{equation}
\label{det:har-a}    
\end{Lem}

The number $z_n$ has an explicit expression by means of the associated operator \cite[Theorem 3]{Ko14}.   

\begin{Lem}  
For integers $n\ge 1$, 
$$
z_n=\sum_{k=1}^{n}\sum_{i_1+\cdots+i_k=n\atop i_1,\dots,i_k\ge m}x_{i_1}\cdots x_{i_k}\,. 
$$ 
\label{ex:har-a} 
\end{Lem}

In order to describe this special case more precisely, consider the geometric sequence of the associated numbers, $x_n=a^{n-m}b$ ($n\ge m$) with $x_0=1$ and $x_1=\cdots=x_{m-1}=0$, where $a$ and $b$ are nonzero integers.  Then by the Cameron's operator, we have the sequence of numbers $z_n$, satisfying  
$$
z_n=a z_{n-1}+b z_{n-m}\quad(n\ge m)\quad\hbox{with}\quad z_0=1\quad\hbox{and}\quad z_1=\cdots=z_{m-1}=0\,.  
$$ 
That is,  the right-hand side of (\ref{cameron-am}) is equal to 
$$ 
\frac{1-a t}{1-a t-b t^m}\,. 
$$ 
Since the number of the elements in the set $\{(i_1,\dots,i_k)|i_1+\cdots+i_k=n,~i_1,\dots,i_k\ge m\}$ is equal to 
$$
\binom{n-k m+k-1}{k-1}\,,  
$$ 
we have 
\begin{align*}  
z_n&=\sum_{k=1}^{n}\sum_{i_1+\cdots+i_k=n\atop i_1,\dots,i_k\ge m}a^{i_1-m}b\cdots a^{i_k-m}b\\
&=\sum_{k=1}^{\fl{\frac{n}{m}}}\sum_{i_1+\cdots+i_k=n\atop i_1,\dots,i_k\ge m}a^{n-k m}b^k\\
&=\sum_{k=1}^{\fl{\frac{n}{m}}}\binom{n-k m+k-1}{k-1}a^{n-k m}b^k\,. 
\end{align*}

\begin{Prop}  
For integers $a$, $b$, $n$ and $m$ with $a\ne 0$, $b\ne 0$ and $n\ge m\ge 1$, we have 
$$
z_n=\sum_{k=1}^{\fl{\frac{n}{m}}}\binom{n-k m+k-1}{k-1}a^{n-k m}b^k\,. 
$$
\label{ex:har-a-ab} 
\end{Prop}

\noindent 
{\it Remark.}  
For initial values $z_n=a^{n-m}b$ ($m\le n\le 2 m-1$), $z_n=a^{n-m}b+(n-2 m+1)a^{n-2 m}b^2$ ($2 m\le n\le 3 m-1$) and $z_n=a^{n-m}b+(n-2 m+1)a^{n-2 m}b^2+\binom{n-3 m+2}{2}a^{n-3 m}b^3$ ($3 m\le n\le 4 m-1$).

In particular, when $a=b=1$, we can get a more explicit expression.  

\begin{Cor}  
For integers $n$ and $m$ with $n\ge m\ge 1$, we have 
$$
z_n=\sum_{k=1}^{\fl{\frac{n}{m}}}\binom{n-k m+k-1}{k-1}\,. 
$$
\label{ex:har-a-11}   
\end{Cor} 

\noindent 
{\it Remark.}  
When $m=1$ in Corollary \ref{ex:har-a-11}, we have 
$$
\sum_{k=1}^{n}\binom{n-1}{k-1}=2^{n-1}\,. 
$$ 
When $m=2$ in Corollary \ref{ex:har-a-11}, we have 
$$
\sum_{k=1}^{\fl{\frac{n}{2}}}\binom{n-k-1}{k-1}=F_{n-1}
$$
(\cite[Theorem 12.4]{Koshy}).  


If the sequence of the associated numbers $x_n$ forms the arithmetic progression, that is, $x_n=(n-m)a+b$ ($n\ge m$) with $x_0=1$ and $x_1=\cdots=x_{m-1}=0$, where $a$ and $b$ are nonzero integers, then by the Cameron's operator, for $m\ge 3$ we have the sequence of numbers $z_n$, satisfying  
$$
z_n=2 z_{n-1}-z_{n-2}+b z_{n-m}+(a-b)z_{n-m-1}\quad(n\ge m+1)
$$ 
with $z_0=1$, $z_1=\cdots=z_{m-1}=0$ and $z_m=b$.  
That is,  the right-hand side of (\ref{cameron-am}) is equal to 
$$ 
\frac{1-2 t+t^2}{1-2 t+t^2-b t^m+(b-a)t^{m+1}}\,. 
$$ 
When $m=2$, $z_n$'s satisfy the recurrence relation 
$$
z_n=2 z_{n-1}+(b-1)z_{n-2}+(a-b)z_{n-3}\quad(n\ge 3) 
$$ 
with $z_0=1$, $z_1=0$ and $z_2=b$.  
The right-hand side of (\ref{cameron-am}) is equal to 
$$ 
\frac{1-2 t+t^2}{1-2 t-(b-1)t^2-(b-a)t^{3}}\,.  
$$ 
When $m=1$, $z_n$'s satisfy the recurrence relation (\ref{example1}), as the original operator mentioned in the Introduction.


By applying Trudi's formula and the inversion formula to Lemma \ref{det:har-a}, we obtain the following expressions for associated numbers $z_n$. 

\begin{Lem}  
For $n\ge 1$, we have 
\begin{equation}
z_n=\sum_{m t_m+(m+1)t_{m+1}+\cdots+n t_n=n}\binom{t_m+t_{m+1}+\cdots+t_m}{t_m,t_{m+1}\dots,t_n}
x_m^{t_m}x_{m+1}^{t_{m+1}}\cdots x_n^{t_n}\,, 
\label{th600}  
\end{equation}    
\begin{equation} 
\left|\begin{array}{ccccc}
z_1&1&0&&\\
z_2&z_1&&&\\
\vdots&\vdots&\ddots&1&0\\
z_{n-1}&z_{n-2}&\cdots&z_1&1\\
z_{n}&z_{n-1}&\cdots&z_2&z_1 
\end{array}\right|
=\begin{cases}
(-1)^{n-1}x_n&(n\ge m)\\
0&(1\le n\le m-1)
\end{cases}\,,
\label{th700}  
\end{equation}  
and 
\begin{align} 
&\sum_{k=1}^{n}(-1)^{k-1}\sum_{i_1+\cdots+i_k=n\atop i_1,\dots,i_k\ge 1}z_{i_1}\cdots z_{i_k}\notag\\
&=\sum_{t_1+2 t_2+\cdots+n t_n=n}\binom{t_1+\cdots+t_n}{t_1,\dots,t_n}
(-1)^{t_1+\cdots+t_n-1}z_1^{t_1}\cdots z_n^{t_n}\notag\\
&=\begin{cases}
x_n&(n\ge m)\\
0&(1\le n\le m-1)
\end{cases}\,. 
\label{inversion:har-a}
\end{align} 
\label{associate-lem}
\end{Lem}

\section{Applications to hypergeometric numbers}   

Cameron's operator can be applied to the sequences of fractional numbers too.   
Recently, incomplete Bernoulli numbers \cite{KLM} are introduced and studied as one kind of generalizations of the classical Bernoulli numbers.  Similarly, incomplete Cauchy numbers \cite{KMS} are introduced and studied as one kind of generalizations of the classical Cauchy numbers. Both generalizations are based upon the incomplete Stirling numbers of both kinds.  
However, these incomplete Bernoulli and Cauchy numbers cannot be generalized as hypergeometric numbers by using the hypergeometric functions. 

Therefore, we introduce modified hypergeometric numbers.  
For integers $N\ge 1$, $n\ge 0$ and $m\ge 1$, define {\it modified restricted hypergeometric Bernoulli numbers} $B_{N,n,\le m}$ by 
\begin{equation}  
\left(\sum_{n=0}^{m}\frac{x^n}{(N+1)^{(n)}}\right)^{-1}=\sum_{n=0}^\infty B_{N,n,\le m}^\ast\frac{x^n}{n!}\,,
\label{mrhber:def}
\end{equation}
where $(x)^{(n)}:=x(x+1)\cdots(x+n-1)$ ($n\ge 1$) is the rising factorial with $(x)^{(0)}=1$. 
When $N=1$, $B_{n,\le m}^\ast=B_{1,n,\le m}^\ast$ are modified restricted Bernoulli numbers (\cite{KR}).   
When $m\to\infty$, $B_{N,n}=B_{N,n,\le\infty}^\ast$ are the original hypergeometric Bernoulli numbers (\cite{HN1,HN2,Ho1,Ho2,HK,Kamano2}) defined by the generating function
\begin{equation}
\frac{1}{{}_1 F_1(1;N+1;x)}=\frac{x^N/N!}{e^x-\sum_{n=0}^{N-1}x^n/n!}=\sum_{n=0}^\infty B_{N,n}\frac{x^n}{n!}\,, 
\label{def:hgb}
\end{equation}
where ${}_1 F_1(a;b;z)$ is the confluent hypergeometric function defined by the generating function
$$
{}_1 F_1(a;b;z)=\sum_{n=0}^\infty\frac{(a)^{(n)}}{(b)^{(n)}}\frac{z^n}{n!}\,.
$$ 
When $N=1$ and $m\to\infty$, $B_n=B_{1,n,\le\infty}^\ast$ are the classical Bernoulli numbers defined by the generating function
$$
\frac{x}{e^x-1}=\sum_{n=0}^\infty B_n\frac{x^n}{n!} 
$$  
with $B_1=-1/2$.  

On the other hand, for integers $N\ge 1$, $n\ge 0$ and $m\ge 1$, define {\it modified associated hypergeometric Bernoulli numbers} $B_{N,n,\ge m}^\ast$ by 
\begin{equation}  
\left(1+\sum_{n=m}^\infty\frac{x^n}{(N+1)^{(n)}}\right)^{-1}=\sum_{n=0}^\infty B_{N,n,\ge m}^\ast\frac{x^n}{n!}\,. 
\label{mahber:def}
\end{equation}
When $N=1$, $B_{n,\ge m}^\ast=B_{1,n,\ge m}^\ast$ are modified associated Bernoulli numbers (\cite{KR}).   
When $m=1$, $B_{N,n}=B_{N,n,\ge 1}^\ast$ are the original hypergeometric Bernoulli numbers. 
When $N=1$ and $m=1$, $B_n=B_{1,n,\ge 1}^\ast$ are the classical Bernoulli numbers.


For integers $N\ge 1$, $n\ge 0$ and $m\ge 1$, define {\it modified restricted hypergeometric Cauchy numbers} $c_{N,n,\le m}^\ast$ by 
\begin{equation}  
\left(N\sum_{n=0}^{m}\frac{(-x)^n}{N+n}\right)^{-1}=\sum_{n=0}^\infty c_{N,n,\le m}^\ast\frac{x^n}{n!}\,,
\label{mrhcau:def} 
\end{equation}
When $N=1$, $c_{n,\le m}^\ast=c_{1,n,\le m}^\ast$ are modified restricted Cauchy numbers (\cite{KR}).   
When $m\to\infty$, $c_{N,n}=c_{N,n,\le\infty}^\ast$ are the original hypergeometric Cauchy numbers defined by the generating function
\begin{equation}  
\frac{1}{{}_2 F_1(1,N;N+1;-x)}=\frac{(-1)^{N-1}x^N/N}{\log(1+t)-\sum_{n=1}^{N-1}(-1)^{n-1}x^n/n}=\sum_{n=0}^\infty c_{N,n}\frac{x^n}{n!}\,,  
\label{def:hgc}
\end{equation}     
where ${}_2 F_1(a,b;c;z)$ is the Gauss hypergeometric function defined by the generating function
$$
{}_2 F_1(a,b;c;z)=\sum_{n=0}^\infty\frac{(a)^{(n)}(b)^{(n)}}{(c)^{(n)}}\frac{z^n}{n!}\,.  
$$ 
When $N=1$ and $m\to\infty$, $c_n=c_{1,n,\le\infty}^\ast$ are classical Cauchy numbers 
defined by the generating function
$$
\frac{x}{\log(1+x)}=\sum_{n=0}^\infty c_n\frac{x^n}{n!}\,. 
$$ 

On the other hand, for integers $N\ge 1$, $n\ge 0$ and $m\ge 1$, define {\it modified associated hypergeometric Cauchy numbers} $c_{N,n,\ge m}^\ast$ by 
\begin{equation}  
\left(1+N\sum_{n=m}^\infty\frac{(-x)^n}{N+n}\right)^{-1}=\sum_{n=0}^\infty c_{N,n,\ge m}^\ast\frac{x^n}{n!}\,. 
\label{mahcau:def}
\end{equation}
When $N=1$, $c_{n,\ge m}^\ast=c_{1,n,\ge m}^\ast$ are modified associated Cauchy numbers (\cite{KR}).   
When $m=1$, $c_{N,n}=c_{N,n,\ge 1}^\ast$ are the original hypergeometric Cauchy numbers. 
When $N=1$ and $m=1$, $c_n=c_{1,n,\ge 1}^\ast$ are the classical Cauchy numbers. 


For integers $N\ge 0$, $n\ge 0$ and $m\ge 1$, define {\it modified restricted hypergeometric Euler numbers} $E_{N,n,\le m}$ by 
\begin{equation}  
\left(\sum_{n=0}^{m}\frac{x^{2 n}}{(2 N+1)^{(2 n)}}\right)^{-1}=\sum_{n=0}^\infty E_{N,n,\le m}^\ast\frac{x^n}{n!}\,. 
\label{mrheu:def}
\end{equation} 
Note that $E_{N,n\le m}^\ast=0$ when $n$ is odd.  
When $N=0$, $E_{n,\le m}^\ast=E_{0,n,\le m}^\ast$ are modified restricted Euler  numbers.   
When $m\to\infty$, $E_{N,n}=E_{N,n,\le\infty}^\ast$ are the original hypergeometric Euler numbers (\cite{Ko7,KZ}) defined by the generating function
\begin{equation}
\frac{1}{{}_1 F_2(1;N+1,(2 N+1)/2;x^2/4)}=\frac{x^{2 N}/(2 N)!}{\cosh x-\sum_{n=0}^{N-1}x^{2 n}/(2 n)!}=\sum_{n=0}^\infty E_{N,n}\frac{x^n}{n!}\,, 
\label{def:hge}
\end{equation}
where ${}_1 F_2(a;b,c;z)$ is the hypergeometric function defined by the generating function
$$
{}_1 F_2(a;b,c;z)=\sum_{n=0}^\infty\frac{(a)^{(n)}}{(b)^{(n)}(c)^{(n)}}\frac{z^n}{n!}\,.
$$ 
When $N=0$ and $m\to\infty$, $E_n=E_{0,n,\le\infty}^\ast$ are the classical Euler numbers defined by the generating function
$$
\frac{1}{\cosh x}=\sum_{n=0}^\infty E_n\frac{x^n}{n!}\,. 
$$  

On the other hand, for integers $N\ge 0$, $n\ge 0$ and $m\ge 1$, define {\it modified associated hypergeometric Euler numbers} $E_{N,n,\ge m}^\ast$ by 
\begin{equation}  
\left(1+\sum_{n=m}^\infty\frac{x^{2 n}}{(2 N+1)^{(2 n)}}\right)^{-1}=\sum_{n=0}^\infty E_{N,n,\ge m}^\ast\frac{x^n}{n!}\,. 
\label{maheu:def}
\end{equation}
Note that $E_{N,n,\ge m}^\ast=0$ when $n$ is odd.  
When $N=0$, $E_{n,\ge m}^\ast=E_{0,n,\ge m}^\ast$ are modified associated Euler numbers.   
When $m=1$, $E_{N,n}=E_{N,n,\ge 1}^\ast$ are the original hypergeometric Euler numbers. 
When $N=0$ and $m=1$, $E_n=E_{0,n,\ge 1}^\ast$ are the classical Euler numbers. 


For integers $N\ge 0$, $n\ge 0$ and $m\ge 1$, define {\it modified restricted hypergeometric Euler numbers of the second kind} $\widehat E_{N,n,\le m}^\ast$ by 
\begin{equation}  
\left(\sum_{n=0}^{m-1}\frac{x^{2 n}}{(2 N+2)^{(2 n)}}\right)^{-1}=\sum_{n=0}^\infty\widehat E_{N,n,\le m}^\ast\frac{x^n}{n!}\,. 
\label{mrheu2:def}
\end{equation} 
Note that $\widehat E_{N,n,\le m}^\ast=0$ when $n$ is odd. 
When $N=0$, $\widehat E_{n,\le m}^\ast=\widehat E_{0,n,\le m}^\ast$ are modified restricted Euler numbers of the second kind.   
When $m\to\infty$, $\widehat E_{N,n}=\widehat E_{N,n,\le\infty}^\ast$ are the original hypergeometric Euler numbers of the second kind (\cite{Ko7,KZ}) defined by the generating function
\begin{align}
\frac{1}{{}_1 F_2(1;N+1,(2 N+3)/2;x^2/4)}&=\frac{x^{2 N+1}/(2 N+1)!}{\sinh x-\sum_{n=0}^{N-1}x^{2 n+1}/(2 n+1)!}\notag\\
&=\sum_{n=0}^\infty\widehat E_{N,n}\frac{x^n}{n!}\,.
\label{def:hge2}
\end{align} 
When $N=0$ and $m\to\infty$, $\widehat E_n=\widehat E_{0,n,\le\infty}^\ast$ are the Euler numbers of the second kind defined by the generating function
$$
\frac{x}{\sinh x}=\sum_{n=0}^\infty\widehat E_n\frac{x^n}{n!}\,. 
$$

On the other hand, for integers $N\ge 0$, $n\ge 0$ and $m\ge 1$, define {\it modified associated hypergeometric Euler numbers of the second kind} $\widehat E_{N,n,\ge m}^\ast$ by 
\begin{equation}  
\left(1+\sum_{n=m}^\infty\frac{x^{2 n}}{(2 N+2)^{(2 n)}}\right)^{-1}=\sum_{n=0}^\infty\widehat E_{N,n,\ge m}^\ast\frac{x^n}{n!}\,. 
\label{maheu:def}
\end{equation}
Note that $\widehat E_{N,n,\ge m}^\ast=0$ when $n$ is odd.  
When $N=0$, $\widehat E_{n,\ge m}^\ast=\widehat E_{0,n,\ge m}^\ast$ are modified associated Euler numbers of the second kind.   
When $m=1$, $\widehat E_{N,n}=\widehat E_{N,n,\ge 1}^\ast$ are the hypergeometric Euler numbers of the second kind. 
When $N=0$ and $m=1$, $\widehat E_n=\widehat E_{0,n,\ge 1}^\ast$ are the Euler numbers of the second kind.   


There are many kinds of generalizations of Bernoulli, Cauchy and Euler numbers, including poly numbers, multiple numbers, Apostol numbers and various $q$-Bernoulli numbers, $p$ adic Bernoulli numbers.   
However, modified restricted hypergeometric Bernoulli, Cauchy and Euler numbers and modified associated hypergeometric Bernoulli, Cauchy and Euler numbers have natural extensions in terms of determinant expressions.  


Applying Lemma \ref{det:har} and Lemma \ref{det:har-a} to Bernoulli, Cauchy and Euler numbers yields the following determinant expressions, respectively. 

For convenience, put for $n\ge 0$
$$
\xi_n:=\begin{cases} 
(-1)^n n!&\text{if $A_{n,\le m}^\ast=B_{N,n,\le m}^\ast$ or $A_{n,\ge m}^\ast=B_{N,n,\ge m}^\ast$}\\ 
n!&\text{if $A_{n,\le m}^\ast=c_{N,n,\le m}^\ast$ or $A_{n,\ge m}^\ast=c_{N,n,\ge m}^\ast$}\\
(-1)^n(2 n)!&\text{if $A_{n,\le m}^\ast=E_{N,n,\le m}^\ast$, $A_{n,\ge m}^\ast=E_{N,n,\ge m}^\ast$}\\ 
\phantom{(-1)^n(2 n)!}&\text{if $A_{n,\le m}^\ast=\widehat E_{N,n,\le m}^\ast$ or $A_{n,\ge m}^\ast=\widehat E_{N,n,\ge m}^\ast$}
\end{cases}
$$ 
and for $j\ge 1$ 
$$
\alpha_j:=\begin{cases} 
\dfrac{N!}{(N+j)!}\quad(N\ge 1)&\text{if $A_{n,\le m}^\ast=B_{N,n,\le m}^\ast$ or $A_{n,\ge m}^\ast=B_{N,n,\ge m}^\ast$}\\ 
\dfrac{N}{N+j}\quad(N\ge 1)&\text{if $A_{n,\le m}^\ast=c_{N,n,\le m}^\ast$ or $A_{n,\ge m}^\ast=c_{N,n,\ge m}^\ast$}\\
\dfrac{(2 N)!}{(2 N+2 j)!}\quad(N\ge 0)&\text{if $A_{n,\le m}^\ast=E_{N,n,\le m}^\ast$, $A_{n,\ge m}^\ast=E_{N,n,\ge m}^\ast$}\\ 
\dfrac{(2 N+1)!}{(2 N+2 j+1)!}\quad(N\ge 0)&\text{if $A_{n,\le m}^\ast=\widehat E_{N,n,\le m}^\ast$ or $A_{n,\ge m}^\ast=\widehat E_{N,n,\ge m}^\ast$}. 
\end{cases}
$$ 
We omit $N$ from the notation for brevity and convenience, as the omission causes no confusion. Note that $N\ge 1$ for Bernoulli's and Cauchy's and $N\ge 0$ for Euler's. 

\begin{Prop}  
For integers $n$ and $m$ with $n\ge m\ge 1$, we have 
$$
A_{n,\le m}^\ast=\xi_n\left|\begin{array}{cc} 
\underbrace{ 
\begin{array}{ccc}
\alpha_1&1&0\\
\vdots&&\ddots\\
\alpha_m&&\\
0&\ddots&\\ 
&&\ddots\\
&&0
\end{array}
}_{n-m} 
\underbrace{ 
\begin{array}{ccc}
&&\\
\ddots&&\\
&\ddots&\\
&\ddots&0\\
&&1\\
\alpha_m&\cdots&\alpha_1
\end{array}
}_{m} 
\end{array} 
\right|. 
$$  
\label{th:rhber}
\end{Prop}

\begin{Prop}  
For integers $n$ and $m$ with $n\ge m\ge 1$, we have 
$$
A_{n,\ge m}^\ast=\xi_n\left|\begin{array}{cc} 
\underbrace{ 
\begin{array}{ccc}
0&1&0\\
\vdots&&\ddots\\
0&&\\
\alpha_m&\ddots&\\ 
\vdots&&\ddots\\
\alpha_n&\cdots&\alpha_m
\end{array}
}_{n-m+1} 
\underbrace{ 
\begin{array}{ccc}
&&\\
\ddots&&\\
&\ddots&\\
&\ddots&0\\
&&1\\
0&\cdots&0
\end{array}
}_{m-1} 
\end{array} 
\right|.  
$$
\label{th:ahber}
\end{Prop}
  
\noindent 
{\it Remark.}  
When $n\le m-1$ for $A_{n,\le m}^\ast=B_{N,n,\le m}^\ast$ in Proposition \ref{th:rhber} or when $m=1$ for $A_{n,\ge m}^\ast=B_{N,n,\ge m}^\ast$ in Proposition \ref{th:ahber}, the result is reduced to a determinant expression of hypergeometric Bernoulli numbers. In addition, when $N=1$, we have a determinant expression of Bernoulli numbers (\cite[p. 53]{
Glaisher}).  
When $n\le m-1$ for $A_{n,\le m}^\ast=c_{N,n,\le m}^\ast$ in Proposition \ref{th:rhber} or when $m=1$ for $A_{n,\ge m}^\ast=c_{N,n,\ge m}^\ast$ in Proposition \ref{th:ahber}, the result is reduced to a determinant expression of hypergeometric Cauchy numbers (\cite{KY}). 
In addition, when $N=1$, we have a determinant expression of Cauchy numbers in (\cite[p. 50]{Glaisher}).   

When $n\le m-1$ for $A_{n,\le m}^\ast=E_{N,n,\le m}^\ast$ in Proposition \ref{th:rhber} or when $m=1$ for $A_{n,\ge m}^\ast=E_{N,n,\ge m}^\ast$ in Proposition \ref{th:ahber}, the result is reduced to a determinant expression of hypergeometric Euler numbers (\cite{Ko7,KZ}). 
In addition, when $N=1$, we have a determinant expression of Euler numbers in ({\it cf.} \cite[p. 52]{Glaisher}).   

When $n\le m-1$ for $A_{n,\le m}^\ast=\widehat E_{N,n,\le m}^\ast$ in Proposition \ref{th:rhber} or when $m=1$ for $A_{n,\ge m}^\ast=\widehat E_{N,n,\ge m}^\ast$ in Proposition \ref{th:ahber}, the result is reduced to a determinant expression of hypergeometric Euler numbers of the second kind (\cite{KZ}). 
In addition, when $N=1$, we have a determinant expression of Euler numbers of the second kind (\cite{Ko7,KZ}).  


Applying Lemma \ref{ex:har} yields different explicit expressions of the modified restricted hypergeometric Bernoulli, Cauchy and Euler numbers.  
 
\begin{Prop} 
For positive integers $n$ and $m$, we have  
$$
A_{n,\le m}^\ast=\xi_n\sum_{k=1}^n(-1)^{n-k}\sum_{i_1+\cdots+i_k=n\atop 1\le i_1,\dots,i_k\le m}\alpha_{i_1}\cdots\alpha_{i_k}.  
$$ 
\label{th:ber5} 
\end{Prop}

There exists an alternative expression including binomial coefficients when every $x_n$ ($n\ge 1$) in (\ref{cameron}) is replaced by $-x_n$ with $x_0=1$ as 
\begin{equation}    
1+\sum_{n=1}^\infty z_n t^n=\left(\sum_{n=0}^\infty x_n t^n\right)^{-1}\,.
\label{cameron-neg}  
\end{equation}
Note that $i_1,\dots,i_k$ take value $0$ too.   

\begin{Prop} 
For positive integers $N$, $n$ and $m$, we have  
$$
A_{n,\le m}^\ast=\xi_n\sum_{k=1}^n(-1)^{n-k}\binom{n+1}{k+1}\sum_{i_1+\cdots+i_k=n\atop 0\le i_1,\dots,i_k\le m}\alpha_{i_1}\cdots\alpha_{i_k}.  
$$ 
\label{th:ber6} 
\end{Prop} 

\begin{proof}  
The proof is based upon a more general result, given in the following Lemma.  The restricted case is obtained similarly.    For Cauchy numbers, $x_n$'s are further replaced by $(-1)^n x_n$.  
\end{proof} 

\begin{Lem}  
If the sequence of $z_n$'s is given by (\ref{cameron-neg}), we have for $n\ge 1$
$$
z_n=\sum_{k=1}^n\binom{n+1}{k+1}\sum_{i_1+\cdots+i_k=n\atop i_1,\dots,i_k\ge 0}x_{i_1}\cdots x_{i_k}\,. 
$$  
\label{lem:cameron-neg}
\end{Lem} 
\begin{proof}
From the definition in (\ref{cameron-neg}), we have for $n\ge 1$ 
\begin{align*}  
z_n&=\frac{1}{n!}\left.\frac{d^n}{d t^n}\left(1-\left(1+\sum_{n=1}^\infty x_n t^n\right)^{-1}\right)\right|_{t=0}\\
&=\frac{1}{n!}\left.\frac{d^n}{d t^n}\sum_{l=0}^\infty\left(1-\sum_{n=0}^\infty x_n t^n\right)^l\right|_{t=0}\\
&=\frac{1}{n!}\left.\frac{d^n}{d t^n}\sum_{l=0}^\infty\sum_{k=0}^l\binom{l}{k}\left(-\sum_{n=0}^\infty x_n t^n\right)^k\right|_{t=0}\\
&=\sum_{k=1}^n(-1)^k\sum_{l=k}^n\binom{l}{k}\sum_{i_1+\cdots+i_k=n\atop i_1,\dots, i_k\ge 0}x_{i_1}\cdots x_{i_k}\,.  
\end{align*}
Notice that the term of $k=0$ disappears when $n\ge 1$.  
By using the relation 
$$
\sum_{l=k}^n\binom{l}{k}=\binom{n+1}{k+1}\,,
$$ 
we get the result. 
\end{proof}
\bigskip

Applying Lemma \ref{ex:har-a} yields different explicit expressions of the modified associated hypergeometric Bernoulli numbers. 

\begin{Prop} 
For positive integers $n$ and $m$, we have  
$$
A_{n,\ge m}^\ast=\xi_n\sum_{k=1}^n\sum_{i_1+\cdots+i_k=n\atop i_1,\dots,i_k\ge m}\alpha_{i_1}\cdots\alpha_{i_k}.  
$$ 
\label{th:ber7} 
\end{Prop} 
 
\noindent 
{\it Remark.}  
When $m\to\infty$ in Proposition \ref{th:ber5} or $m=1$ in Proposition \ref{th:ber7}, we have 
$$
B_{N,n}=n!\sum_{k=1}^n\sum_{i_1+\cdots+i_k=n\atop i_1,\dots,i_k\ge 1}\frac{(-N!)^k}{(N+i_1)!\cdots(N+i_k)!}\,,   
$$ 
$$
c_{N,n}=n!\sum_{k=1}^n(-1)^{n-k}\sum_{i_1+\cdots+i_k=n\atop i_1,\dots,i_k\ge 1}\frac{N^k}{(N+i_1)\cdots(N+i_k)} 
$$ 
(\cite{KY}),   
$$
E_{N,2 n}=(2 n)!\sum_{k=1}^n\sum_{i_1+\cdots+i_k=n\atop i_1,\dots,i_k\ge 1}\frac{\bigl(-(2 N)!\bigr)^k}{(2 N+2 i_1)!\cdots(2 N+2 i_k)!}
$$ 
and  
$$
\widehat E_{N,2 n}=(2 n)!\sum_{k=1}^n\sum_{i_1+\cdots+i_k=n\atop i_1,\dots,i_k\ge 1}\frac{\bigl(-(2 N+1)!\bigr)^k}{(2 N+2 i_1+1)!\cdots(2 N+2 i_k+1)!}\,. 
$$ 
(\cite{KZ}).

There exists an alternative expression including binomial coefficients. The results are similarly obtained by using Lemma \ref{lem:cameron-neg} in the associated case. 

\begin{Prop} 
For positive integers $N$, $n$ and $m$, we have  
$$
A_{n,\ge m}^\ast=\xi_n\sum_{k=1}^n\binom{n+1}{k+1}\sum_{i_1+\cdots+i_k=n\atop i_1,\dots,i_k\ge m-1}\alpha_{i_1}\cdots\alpha_{i_k}. 
$$ 
\label{th:cau8} 
\end{Prop} 

\noindent 
{\it Remark.}  
When $m\to\infty$ in Proposition \ref{th:ber6} or $m=1$ in Proposition \ref{th:cau8}, we have 
$$
B_{N,n}=n!\sum_{k=1}^n\binom{n+1}{k+1}\sum_{i_1+\cdots+i_k=n\atop i_1,\dots,i_k\ge 0}\frac{(-N!)^k}{(N+i_1)!\cdots(N+i_k)!}\,, 
$$ 
$$
c_{N,n}=n!\sum_{k=1}^n(-1)^{n-k}\binom{n+1}{k+1}\sum_{i_1+\cdots+i_k=n\atop i_1,\dots,i_k\ge 0}\frac{N^k}{(N+i_1)\cdots(N+i_k)}\,, 
$$  
$$
E_{N,2 n}=(2 n)!\sum_{k=1}^n\binom{n+1}{k+1}\sum_{i_1+\cdots+i_k=n\atop i_1,\dots,i_k\ge 0}\frac{\bigl(-(2 N)!\bigr)^k}{(2 N+2 i_1)!\cdots(2 N+2 i_k)!}  
$$ 
and 
$$
\widehat E_{N,2 n}=(2 n)!\sum_{k=1}^n\binom{n+1}{k+1}\sum_{i_1+\cdots+i_k=n\atop i_1,\dots,i_k\ge 0}\frac{\bigl(-(2 N+1)!\bigr)^k}{(2 N+2 i_1+1)!\cdots(2 N+2 i_k+1)!} 
$$ 
(\cite{KZ}).

\bigskip

By applying Lemma \ref{th300} and Lemma \ref{associate-lem} (\ref{th600}) to Proposition \ref{th:rhber} and Proposition \ref{th:ahber}, we obtain an explicit expression for modified incomplete Bernoulli, Cauchy and Euler numbers. 

\begin{Prop}  
For $n\ge m\ge 1$, we have 
$$ 
A_{n,\le m}^\ast=(-1)^n\xi_n\sum_{t_1+2 t_2+\cdots+m t_m=n}\binom{t_1+\cdots+t_m}{t_1,\dots,t_m}
(-1)^{t_1+\cdots+t_m}\alpha_1^{t_1}\cdots\alpha_m^{t_m}
$$ 
and 
\begin{multline*} 
A_{n,\ge m}^\ast=(-1)^n\xi_n\sum_{m t_m+(m+1)t_{m+1}+\cdots+n t_n=n}\binom{t_m+t_{m+1}+\cdots+t_n}{t_m,t_{m+1},\dots,t_n}\\
\times(-1)^{t_m+t_{m+1}+\cdots+t_n}\alpha_m^{t_m}\alpha_{m+1}^{t_{m+1}}\cdots\alpha_n^{t_n}\,. 
\end{multline*}  
\label{th310}  
\end{Prop}  

By applying Lemma \ref{th400} and Lemma \ref{associate-lem} (\ref{th700}) as the result of the inversion relation to Proposition \ref{th:rhber} and Proposition \ref{th:ahber}, respectively, we have the following.  

\begin{Prop}   
For $n\ge 1$, we have 
$$ 
\left|
\begin{array}{ccccc} 
\frac{A_{1,\le m}^\ast}{\xi_1}&1&&&\\  
\frac{A_{2,\le m}^\ast}{\xi_2}&\frac{A_{1,\le m}^\ast}{\xi_1}&&&\\ 
\vdots&\vdots&\ddots&1&\\ 
\frac{A_{n-1,\le m}^\ast}{\xi_{n-1}}&\frac{A_{n-2,\le m}^\ast}{\xi_{n-2}}&\cdots&\frac{A_{1,\le m}^\ast}{\xi_1}&1\\ 
\frac{A_{n,\le m}^\ast}{\xi_n}&\frac{A_{n-1,\le m}^\ast}{\xi_{n-1}}&\cdots&\frac{A_{2,\le m}^\ast}{\xi_2}&\frac{A_{1,\le m}^\ast}{\xi_1}
\end{array} 
\right|
=\begin{cases}
\alpha_n&\text{($n\le m$)}\\
0&\text{($n>m$)} 
\end{cases}
$$ 
and 
$$ 
\left|
\begin{array}{ccccc} 
\frac{A_{1,\ge m}^\ast}{\xi_1}&1&&&\\  
\frac{A_{2,\ge m}^\ast}{\xi_2}&\frac{A_{1,\ge m}^\ast}{\xi_1}&&&\\ 
\vdots&\vdots&\ddots&1&\\ 
\frac{A_{n-1,\ge m}^\ast}{\xi_{n-1}}&\frac{A_{n-2,\ge m}^\ast}{\xi_{n-2}}&\cdots&\frac{A_{1,\ge m}^\ast}{\xi_1}&1\\ 
\frac{A_{n,\ge m}^\ast}{\xi_n}&\frac{A_{n-1,\ge m}^\ast}{\xi_{n-1}}&\cdots&\frac{A_{2,\ge m}^\ast}{\xi_2}&\frac{A_{1,\ge m}^\ast}{\xi_1}
\end{array} 
\right|
=\begin{cases}
0&\text{($n<m$)}\\
\alpha_n&\text{($n\ge m$)} 
\end{cases} 
$$ 
\label{th410}  
\end{Prop}

By applying the inversion relations to Proposition \ref{th310}, we also have the following. 

\begin{Prop}  
For $n\ge 1$, we have 
\begin{align*} 
\alpha_n&=\sum_{t_1+2 t_2+\cdots+m t_m=n}\binom{t_1+\cdots+t_m}{t_1,\dots,t_m}
(-1)^{n-t_1-\cdots-t_m}\\
&\quad\times\left(\frac{A_{1,\le m}^\ast}{\xi_1}\right)^{t_1}\left(\frac{A_{2,\le m}^\ast}{\xi_2}\right)^{t_2}\cdots\left(\frac{A_{m,\le m}^\ast}{\xi_m}\right)^{t_m}\quad(n\le m)\\
&=\sum_{m t_m+(m+1)t_{m+1}+\cdots+n t_n=n}\binom{t_m+t_{m+1}+\cdots+t_n}{t_m,t_{m+1},\dots,t_n}
(-1)^{n-t_m-t_{m+1}-\cdots-t_n}\\
&\quad\times\left(\frac{A_{m,\ge m}^\ast}{\xi_m}\right)^{t_m}\left(\frac{A_{m+1,\ge m}^\ast}{\xi_{m+1}}\right)^{t_{m+1}}\cdots\left(\frac{A_{n,\ge m}^\ast}{\xi_n}\right)^{t_n}\quad(n\ge m)\,. 
\end{align*}   
\label{inversion:bel}
\end{Prop}

\section*{Acknowledgement}  

We thank the anonymous referee for careful reading of our manuscript and many insightful comments and suggestions.

\end{document}